\documentclass[draft,reqno]{amsproc}
\usepackage{amsfonts}
\usepackage{amssymb,color}

\numberwithin{equation}{section}
\theoremstyle{plain}
\newtheorem{thm}{Theorem}[section]
\newtheorem{pro}[thm]{Proposition}
\newtheorem{lem}[thm]{Lemma}
\newtheorem{cor}[thm]{Corollary}

\theoremstyle{definition}

\newtheorem*{dfn*}{Definition}
\newtheorem{exa}[thm]{Example}

\theoremstyle{remark}
\newtheorem{rem}[thm]{Remark}

\newcommand*{\card}[1]{\mathrm{card}\,#1}
\newcommand*{\cbb}{\mathbb C}
\newcommand*{\D}{\,\mathrm d}
\newcommand*{\Le}{\leqslant}
\newcommand*{\Ge}{\geqslant}
\newcommand*{\ideal}[1]{\mathcal I(#1)}
\newcommand*{\idealc}[1]{\mathcal I_\mathsf{\hspace{.25ex}c}(#1)}
\newcommand*{\idealf}[1]{\mathcal I_{#1}}
\newcommand*{\idealfc}[1]{\mathcal I^{\hspace{.25ex}\mathsf{c}}_{#1}}
\newcommand*{\I}{\mathrm{i}}

\newcommand*{\pcal}{\mathcal P}
\newcommand*{\rcal}{\mathcal R}
\newcommand*{\rbb}{\mathbb R}
\newcommand*{\zera}[1]{\mathcal Z(#1)}

\begin{document}

   \title[Algebraic sets of types A, B, and C coincide]
{Algebraic sets of types A, B, and C coincide}

   \author{Torben Maack Bisgaard}


   \address{Nandrupsvej 7 st.\ th., DK-2000 Frederiksberg,
Denmark}

   \email{torben.bisgaard@get2net.dk}

   \author{Jan Stochel}

   \address{Instytut Matematyki, Uniwersytet Jagiello\'nski,
ul. Reymonta 4, PL-30059 Krak\'ow} \email{stochel@im.uj.edu.pl}

\thanks{Running expenses of
the first author were covered by the Carlsberg
Foundation. The second author has been supported by
the MNiSzW grant N201 026 32/1350.}




   \begin{abstract} It is proved that the definition of an
algebraic set of type {\sf A} (a notion related to the
multidimensional Hamburger moment problem) does not
depend on the choice of a polynomial describing the
algebraic set in question and that an algebraic set of
type {\sf B} is always of type {\sf A}. This answers
in the affirmative two questions posed in 1992 by the
second author. It is also shown that an algebraic set
is of type {\sf A} if and only if it is of type {\sf
C} (a notion linked to orthogonality of polynomials of
several variables). This, in turn, enables us to
answer three questions posed in 2005 by Cicho\'n,
Stochel, and Szafraniec.
   \end{abstract}
   \maketitle
   \section{Introduction} It is well known that
each Hamburger moment $n$-sequence on an algebraic subset
$V$ of $\rbb^n$ is positive definite and satisfies a
certain system of algebraic equations generated by $V$. If
the converse holds, the algebraic set $V$ is called of {\em
type} {\sf A} or {\sf B}, depending on the system of
equations (cf.\ \cite{sto}). The knowledge of various
classes of algebraic sets of type {\sf A} seems to be of
great importance because of applications, not only in
moment problems, but also in the theory of orthogonal
polynomials in several variables (cf.\ \cite{c-s-sz}). We
now provide a brief overview of some known classes of
algebraic sets of type {\sf A}. First, note that any
algebraic subset $V$ of $\rbb$ is of type {\sf A}. Indeed,
the case of $V = \rbb$ follows from Hamburger's theorem
(cf.\ \cite{b-ch-r}), while the case of $V
\varsubsetneq\rbb$ can be inferred from \cite[Lemma
49]{c-s-sz} (or \cite[Theorem 43]{c-s-sz}). In turn, each
compact algebraic subset of $\rbb^n$ is of type {\sf A}
(see \cite[Theorem 43]{c-s-sz} which heavily depends on
\cite[Theorem 1]{sch}). By \cite[Theorem 5.4]{sto}, every
algebraic subset of $\rbb^2$ induced by a non-zero
polynomial $P \in \rbb[X_1,X_2]$ of degree at most $2$ is
of type {\sf A}. However, as shown in \cite[Theorem
6.3]{sto}, there exists an algebraic subset $V$ of $\rbb^2$
induced by a polynomial $P \in \rbb[X_1,X_2]$ of degree $3$
(e.g.\ $P= X_2(X_2 - X_1^2)$) such that $V$ is not of type
{\sf A} (using a quotient technique from Section \ref{s5},
the careful reader can deduce from \cite{na-sa} that the
polynomial $P=X_1^3-X_2^2$ induces an algebraic subset of
$\rbb^2$ which is not of type {\sf A}). The question of
when an algebraic subset of $\rbb^n$ induced by a
polynomial $P \in \rbb[X_1, \ldots, X_n]$ of the form
$P=X_1^{\alpha_1} \cdots X_n^{\alpha_n} - X_1^{\beta_1}
\cdots X_n^{\beta_n}$ ($\alpha_1, \ldots, \alpha_n$ and
$\beta_1, \ldots, \beta_n$ are nonnegative integers) is of
type {\sf A} has been completely answered in \cite{bis}.
The interested reader is referred to
\cite{rie,hav1,hav2,dev,cho,ak,sch0,mcg,fug,sto,st-sz,bis,st-sz2,st-sz3,c-s-sz}
for further information concerning this subject including
other classes of sets of type {\sf A} (see also
\cite{atz,mas,cas,sch,c-p,pv,c-f,c-f2,pow-sch,sch2,fial,pow1,c-f3,c-f4}
for moments on semi-algebraic sets and related problems).

In this paper we show that the definition of type {\sf A}
is correct and that types {\sf A} and {\sf B} coincide
(cf.\ Theorems \ref{popr} and \ref{A=B} as well as Theorems
\ref{popr+} and \ref{A=B+}). This answers in the
affirmative two questions posed in \cite{sto}. The main
tool in our considerations is the Real Nullstellensatz, the
celebrated theorem of real algebraic geometry. A
substantial part of the paper is devoted to the study of
positive definite functionals vanishing on polynomial
ideals.

The notion of type {\sf C} was invented in
\cite{c-s-sz} to distinguish a class of polynomial
ideals which have the property that each sequence of
polynomials in several variables satisfying an
appropriate three term recurrence relation modulo the
polynomial ideal is orthogonal with respect to some
positive measure. We prove that an algebraic subset
$V$ of $\rbb^n$ is of type {\sf A} if and only if the
set ideal $\idealc{V}$ is of type {\sf C} (cf.\
Theorem \ref{typC=A}), which answers in the
affirmative Question 2 posed in \cite{c-s-sz}. In
fact, Theorem \ref{typC=A} enables us to answer in the
negative two more questions (Questions 3 and 4) posed
in \cite{c-s-sz}. Namely, we show that there exists a
non-zero set ideal of $\cbb[X_1, \ldots, X_n]$ ($n \Ge
2$) which is not of type {\sf C} and that the zero
ideal of $\cbb[X_1, \ldots, X_n]$ ($n \Ge 2$) is not
of type {\sf C}. Let us point out that \cite[Question
4]{c-s-sz} was answered in the negative for $N=2$ by
Friedrich \cite{fr}; he constructed a strictly
positive definite linear functional on $\cbb[X_1,
X_2]$ which is not a moment functional, showing that
the zero ideal of $\cbb[X_1, X_2]$ is not of type {\sf
C}. The proof of Theorem \ref{typC=A} depends on a
quotient technique for finitely generated commutative
(real or complex) $*$-algebras with unit worked out in
Section \ref{s5}.

Types {\sf A}, {\sf B}, and {\sf C} are shown to be
equivalent also in the context of the multidimensional
complex moment problem (cf.\ Section \ref{s7}). The
paper ends with Appendix which contains complex
counterparts of Lemmata \ref{realradical} and
\ref{rn}.

Let us point out that in Sections $2$ to $5$ we use
the notation $(P)$, $\ideal Z$ and $\idealf
\varLambda$ only in the case of the real algebra
$\rbb[X_1, \ldots, X_n]$, while in Sections \ref{s5},
\ref{s6}, and \ref{s7} the same notation is used also
in the case of the complex $*$-algebra $\cbb[X_1,
\ldots, X_n]$. Since it is clear from the context
which algebra is considered, there is no clash of
notation. In Sections \ref{s3} and \ref{s4}, and
Appendix \ref{s8} the complexified notation
$(P)_\mathsf c$, $\idealc{Z}$ and $\idealfc\varLambda$
is employed in order to distinguish it from the real
one in the case when both are used simultaneously.
   \section{\label{s1}Positive definite
functionals vanishing on ideals}
   In what follows $\rbb$ and $\cbb$ stand for the fields
of real and complex numbers respectively. Let $n$ be a
positive integer. Denote by $\rbb[X_1, \ldots, X_n]$ the
real algebra of all polynomials in $n$ commuting
indeterminates $X_1, \dots, X_n$ with real coefficients.
Given a polynomial $P \in \rbb[X_1, \ldots, X_n]$, a subset
$Z$ of $\rbb^n$ and an ideal $I$ of $\rbb[X_1, \ldots,
X_n]$, we write
   \begin{align}
(P) & = \text{ the ideal of $\rbb[X_1, \ldots, X_n]$
generated by } P, \notag
   \\
   \ideal{Z} & = \{P \in \rbb[X_1, \ldots, X_n] \colon
P(x)=0 \text{ for every } x \in Z\}, \notag
   \\
\zera{I} & = \{x \in \rbb^n \colon P(x)=0 \text{ for every
} P \in I\}. \label{zera}
   \end{align}
The following properties of $\zera{\cdot}$ and
$\ideal{\cdot}$ are easy and elementary to prove.
   \begin{lem} \label{dodat}
If $I$ and $J$ are ideals of $\rbb[X_1, \ldots, X_n]$ such
that $I \subset J$, then $I \subset \ideal{\zera{I}}
\subset \ideal{\zera{J}}$. If $Z \subset \rbb^n$, then $Z
\subset \zera{\ideal{Z}}$ and
$\ideal{Z}=\ideal{\zera{\ideal Z}}$.
   \end{lem}
   An ideal $I$ of $\rbb[X_1, \ldots, X_n]$ is said to be a
{\em set ideal} if there exists a subset $Z$ of $\rbb^n$
such that $I = \ideal{Z}$. It is worth noting that there
are ideals which are not set ideals (e.g.\ $I=(X_1^s)$,
where $s \Ge 2$ is an integer). A set of the form $\zera
I$, where $I$ is an ideal of $\rbb[X_1, \ldots, X_n]$, is
called a {\em $($real\/$)$ algebraic set}. It is known that
for each algebraic subset $V$ of $\rbb^n$ there exists a
polynomial $P \in \rbb[X_1, \ldots, X_n]$ such that $V =
\mathcal Z_P := \{x \in \rbb^n \colon P(x)=0\}$ (indeed, by
Hilbert's basis theorem the ideal $I$ is generated by
finitely many polynomials $Q_1, \ldots, Q_m \in \rbb[X_1,
\ldots, X_n]$ and so $\zera I = \mathcal Z_Q$ with $Q=Q_1^2
+ \ldots + Q_m^2$). If $V = \mathcal Z_P$, then we say that
$V$ is an algebraic set {\em induced} by $P$. For
fundamentals of the theory of real algebraic sets we
recommend the monographs \cite{b-r,b-c-r}.

Let $I$ be an ideal of $\rbb[X_1, \ldots, X_n]$. We say
that $I$ is {\em real} if for every finite sequence $Q_1,
\ldots, Q_s \in \rbb[X_1, \ldots, X_n]$ such that $Q_1^2 +
\dots + Q_s^2 \in I$, we have $Q_1, \ldots, Q_s \in I$.
Denote by $\sqrt[{\mathrm R}]I$ the {\em real radical} of
$I$, that is the intersection of all real prime ideals of
$\rbb[X_1, \ldots, X_n]$ containing $I$ (if there is no
such prime ideal, then we put $\sqrt[{\mathrm R}]I =
\rbb[X_1, \ldots, X_n]$).

We now recall a useful description of the real radical.
   \begin{lem}[\mbox{\cite[Proposition 4.1.7]{b-c-r}}]
\label{realradical}
   Let $I$ be an ideal of $\rbb[X_1, \ldots, X_n]$. Then
$\sqrt[{\mathrm R}]I$ is the smallest real ideal of
$\rbb[X_1, \ldots, X_n]$ containing $I$. Moreover, a
polynomial $P \in \rbb[X_1, \ldots, X_n]$ belongs to
$\sqrt[{\mathrm R}]I$ if and only if there exist finitely
many polynomials $Q_1, \ldots, Q_s \in \rbb[X_1, \ldots,
X_n]$ and an integer $m \Ge 0$ such that $P^{2m} + Q_1^2 +
\dots + Q_s^2 \in I$.
   \end{lem}
The following variant of the Real Nullstellensatz
plays a crucial role in the present paper.
   \begin{lem}[\mbox{\cite[Corollary 4.1.8]{b-c-r}}] \label{rn}
If $I$ is an ideal of $\rbb[X_1, \ldots, X_n]$, then
   \begin{align*}
\ideal{\zera{I}} = \sqrt[{\mathrm R}]I.
   \end{align*}
   \end{lem}
A linear functional $\varLambda \colon \rbb[X_1,
\ldots, X_n] \to \rbb$ is said to be {\em positive
definite} if $\varLambda(P^2) \Ge 0$ for every $P \in
\rbb[X_1, \ldots, X_n]$. Note that every positive
definite linear functional $\varLambda \colon
\rbb[X_1, \ldots, X_n] \to \rbb$ satisfies the Schwarz
inequality
   \begin{align} \label{schw}
|\varLambda(PQ)| \Le \varLambda(P^2)^{1/2} \varLambda(Q^2)^{1/2},
\quad P,Q \in \rbb[X_1, \ldots, X_n].
   \end{align}
Given a positive definite linear functional $\varLambda \colon
\rbb[X_1, \ldots, X_n] \to \rbb$, we write
   \begin{align*}
\idealf\varLambda = \{P \in \rbb[X_1, \ldots, X_n]
\colon \varLambda(P^2)=0\}.
   \end{align*}
We now state three basic properties of the set
$\idealf\varLambda$.
   \begin{lem} \label{cor}
If $\varLambda \colon \rbb[X_1, \ldots, X_n] \to \rbb$
is a positive definite linear functional,
then\footnote{\;It follows from (ii) that
$\idealf\varLambda = \sqrt {\idealf\varLambda}$
because $\label{nilradreal} I \subset \sqrt I \subset
\sqrt[{\mathrm R}]I$ for every ideal $I$ of $\rbb[X_1,
\ldots, X_n]$, $\sqrt I$ being the radical of $I$ (see
\cite{hun} for the basic properties of the radical).}
   \begin{enumerate}
   \item[(i)] $\idealf\varLambda$ is the greatest ideal of
$\rbb[X_1, \ldots, X_n]$ contained in the kernel of
$\varLambda$,
   \item[(ii)] $\idealf\varLambda =
\sqrt[{\mathrm R}]{\idealf\varLambda}$,
   \item[(iii)] $\idealf\varLambda$ is a set ideal.
   \end{enumerate}
   \end{lem}
   \begin{proof}
(i) Since the mapping $P \mapsto \varLambda(P^2)^{1/2}$ is
a seminorm on $\rbb[X_1, \ldots, X_n]$, we deduce that
$\idealf\varLambda$ is a vector space. Using \eqref{schw},
we get
   \begin{align*}
|\varLambda((PQ)^2)| \Le \varLambda(P^2)^{1/2}
\varLambda(P^2Q^4)^{1/2} = 0, \quad P \in
\idealf\varLambda, \, Q \in \rbb[X_1, \ldots, X_n],
   \end{align*}
which shows that $\idealf\varLambda$ is an ideal of
$\rbb[X_1, \ldots, X_n]$. Applying \eqref{schw} to $P
\in \idealf\varLambda$ and $Q=1$, we see that
$\idealf\varLambda$ is contained in the kernel $\ker
\varLambda$ of $\varLambda$. Moreover, if $J$ is an
ideal of $\rbb[X_1, \ldots, X_n]$ such that $J \subset
\ker \varLambda$, then for every $P \in J$, $P^2 \in J
\subset \ker \varLambda$ and consequently $P \in
\idealf\varLambda$; hence $J \subset
\idealf\varLambda$, which proves (i).

(ii) By Lemma \ref{realradical}, it is enough to show
that the ideal $\idealf\varLambda$ is real. If $Q_1,
\ldots, Q_s \in \rbb[X_1, \ldots, X_n]$ are such that
$Q_1^2 + \dots + Q_s^2 \in \idealf\varLambda$, then by
(i) we have
   \begin{align*}
0 = \varLambda(Q_1^2 + \dots + Q_s^2) =
\varLambda(Q_1^2) + \dots + \varLambda(Q_s^2),
   \end{align*}
which, when combined with positive definiteness of
$\varLambda$, implies that $Q_1, \ldots, Q_s \in
\idealf\varLambda$.

(iii) is a direct consequence of (ii) and Lemma
\ref{rn}.
   \end{proof}
The next lemma constitutes the main tool for our further
investigations.
   \begin{lem} \label{1}
Let $I$ be an ideal of $\rbb[X_1, \ldots, X_n]$ and let
$\varLambda \colon \rbb[X_1, \ldots, X_n] \to \rbb$ be a positive
definite linear functional. Then $\varLambda$ vanishes on $I$ if
and only if it vanishes on $\ideal{\zera{I}}$.
   \end{lem}
   \begin{proof}  Since $I \subset \ideal{\zera{I}}$,
it suffices to prove the ``only if'' part of the
conclusion. If $\varLambda$ vanishes on $I$, then, by Lemma
\ref{cor}\,(i), we get $I \subseteq \idealf\varLambda$.
This, together with Lemma \ref{dodat} and Lemma \ref{cor}
(iii) and (i), implies that
   \begin{align*}
\ideal{\zera{I}} \subseteq \ideal{\zera{\idealf\varLambda}}
= \idealf\varLambda \subseteq \ker \varLambda,
   \end{align*}
which completes the proof.
   \end{proof}
\section{\label{s2}Types A and B coincide}
A functional $\varLambda \colon \rbb[X_1, \ldots, X_n] \to
\rbb$ is called a {\em moment functional} on a closed
subset $F$ of $\rbb^n$ if there exists a positive Borel
measure $\mu$ on $F$ such that\footnote{\;We tacitly assume
that all polynomials are absolutely integrable with respect
to $\mu$; in particular, $\mu$ has to be finite.}
   \begin{align} \label{lq}
\varLambda(Q) = \int_F Q(x) \text{ d} \mu(x), \quad Q \in
\rbb[X_1, \ldots, X_n].
   \end{align}
A positive Borel measure $\mu$ on $F$ satisfying
\eqref{lq} is called a {\em representing measure} of
$\varLambda$. We shall abbreviate the expression
``moment functional on $\rbb^n$'' simply to ``moment
functional''. A representing measure of a moment
functional may not be unique (cf.\ \cite{b-ch-r}).
Obviously, each moment functional on a closed subset
$F$ of $\rbb^n$ is a moment functional (but not
conversely) and each moment functional is a positive
definite linear functional (but not conversely except
for $n = 1$, cf.\ \cite{b-ch-r,fr,bis2,bis3}).

Consider a polynomial $P \in \rbb[X_1, \ldots, X_n]$.
In view of \cite[Proposition 2.1]{sto}, a moment
functional $\varLambda$ is a moment functional on
$\mathcal Z_P$ if and only if $\varLambda$ vanishes on
the principal ideal $(P)$, or equivalently on the
ideal $\ideal{\mathcal Z_P}$. Hence, each moment
functional on $\mathcal Z_P$ is a positive definite
linear functional vanishing on $\ideal{\mathcal Z_P}$,
and consequently on $(P)$. Following \cite{sto}, we
say that $P$ is of {\em type} {\sf A} (respectively
{\sf B}) if each positive definite linear functional
$\varLambda \colon \rbb[X_1, \ldots, X_n] \to \rbb$
vanishing on $(P)$ (respectively $\ideal{\mathcal
Z_P}$) is a moment functional; any such functional is
automatically a moment functional on $\mathcal Z_P$.
Clearly, each polynomial of type {\sf A} is of type
{\sf B}.
   \begin{thm} \label{popr}
Let $P,Q \in \rbb[X_1, \ldots, X_n]$ be such that $\mathcal Z_P =
\mathcal Z_Q$. Then $P$ is of type {\sf A} if and only if $Q$ is
of type {\sf A}.
   \end{thm}
   \begin{proof}
Suppose that $P$ is of type {\sf A}. Set $I = (P)$ and
$J=(Q)$. Then
   \begin{align} \label{ij}
\ideal{\zera I} = \ideal{\mathcal Z_P} =
\ideal{\mathcal Z_Q} =\ideal{\zera J}.
   \end{align}
If a positive definite linear functional $\varLambda
\colon \rbb[X_1, \ldots, X_n] \to \rbb$ vanishes on
$J$, then by \eqref{ij} and Lemma \ref{1}, the
functional $\varLambda$ vanishes on $\ideal{\zera I}
\supset I$. Since $P$ is of type {\sf A}, $\varLambda$
is a moment functional. Hence $Q$ is of type {\sf A}.
By symmetry the proof is complete.
   \end{proof}
Let us recall the original definitions of algebraic sets
of types {\sf A} and {\sf B} from~ \cite{sto}.
   \begin{dfn*}
An algebraic subset $V$ of $\rbb^n$ induced by a
polynomial $P \in \rbb[X_1, \ldots, X_n]$ is of {\em
type} {\sf A} $($respectively {\sf B}$)$ if $P$ is of
type {\sf A} $($respectively {\sf B}$)$.
   \end{dfn*}
Theorem \ref{popr} implies that the definition of algebraic
sets of type {\sf A} is correct; that is, it shows that the
meaning of the statement ``$V$ is of type $\sf A$'' is
independent of the choice of a polynomial $P$ which induces
the algebraic set $V$. This answers in the affirmative a
question posed in \cite[Section 4]{sto}. Certainly, the
definition of type {\sf B} is correct.

The next theorem is an immediate consequence of Lemma
\ref{1} (with $I = (P)$). It answers in the affirmative
another question posed in \cite[Section 4]{sto}.
   \begin{thm} \label{A=B}
An algebraic subset of $\rbb^n$ is of type {\sf A} if
and only if it is of type~ {\sf B}.
   \end{thm}
   \section{\label{s3}Complexification}
In a recent paper \cite{c-s-sz}, types {\sf A} and
{\sf B} have been ``complexified''. The reason for
this modification comes from the spectral theory of
Hilbert space operators which is well developed in the
complex case. We now discuss this in more detail.

In what follows, members of $\pcal_n := \cbb[X_1,
\ldots,X_n]$ and $\rcal_n := \rbb[X_1, \ldots,X_n]$ will be
identified with complex and real polynomial functions on
$\rbb^n$ respectively. The set $\pcal_n$ is a $*$-algebra
with involution $(P+\I Q)^* := P- \I Q$ for $P,Q \in
\rcal_n$. We say that a complex linear functional
$\varLambda\colon \pcal_n \to \cbb$ is {\em positive
definite} if $\varLambda(P^*P) \Ge 0$ for all $P \in
\pcal_n$. If $\varLambda$ is a positive definite complex
linear functional on $\pcal_n$, then
   \begin{align*}
\varLambda(P^*)=\overline{\varLambda(P)}, \quad P \in
\pcal_n \quad (\text{equivalently: }
\varLambda(\rcal_n) \subset \rbb),
   \end{align*}
and so $\varLambda_{\mathsf r}:= \varLambda|_{\rcal_n}$ is
a positive definite real linear functional on $\rcal_n$.
The mapping $\varLambda \mapsto \varLambda_\mathsf r$ is a
bijection between the set of all positive definite complex
linear functionals on $\pcal_n$ and the set of all positive
definite real linear functionals on $\rcal_n$. Given a
positive definite complex linear functional $\varLambda$ on
$\pcal_n$ and $Z \subset \rbb^n$, we write
   \begin{align}  \label{2w1}
   \begin{aligned}
\idealfc\varLambda & = \{P \in \pcal_n \colon
\varLambda(P^*P)=0\},
   \\
\idealc{Z} & = \{P \in \pcal_n\colon P(x)=0 \text{ for
every } x \in Z\}.
   \end{aligned}
   \end{align}
   Moreover, we denote by $(P)_\mathsf{c}$ the ideal
of $\pcal_n$ generated by $P\in \pcal_n$. We point out
that the sets $\idealf\varLambda$, $\ideal{Z}$ and
$(P)$ defined in Section \ref{s1} are ideals of
$\rcal_n$, the first two being real, and that the sets
$\idealfc\varLambda$ and $\idealc{Z}$ defined in
\eqref{2w1} are ideals of $\pcal_n$ that may not be
real ($\pcal_n$ is the only real ideal of $\pcal_n$).
The following simple fact is stated without proof (we
write $\mathcal Z_P = \{x \in \rbb^n \colon P(x)=0\}$
also for $P \in \pcal_n$).
   \begin{lem} \label{12.XII.06}
Let $P \in \pcal_n$. A positive definite complex
linear functional $\varLambda$ on $\pcal_n$ vanishes
on $(P)_\mathsf c$ $($respectively $\idealc{\mathcal
Z_P}$$)$ if and only if $\varLambda_\mathsf{r}$
vanishes on $(P^*P)$ $($respectively $\ideal{\mathcal
Z_P}$$)$. Moreover, if $P \in \rcal_n$, then $(P^*P)$
can be replaced in the above equivalence by $(P)$.
   \end{lem}
   We now formulate ``complexified'' variants of
Lemmata \ref{cor} and \ref{1} (see also Appendix for
the complex counterparts of Lemmata \ref{realradical}
and \ref{rn}).
   \begin{lem} \label{11.XII.06}
If $\varLambda$ is a positive definite complex linear
functional on $\pcal_n$, then
   \begin{enumerate}
   \item[(i)] $\idealfc\varLambda$ is the greatest ideal of
$\pcal_n$ contained in the kernel of $\varLambda$,
   \item[(ii)] $\idealfc\varLambda$ is a set ideal
$($i.e.\ $\idealfc\varLambda = \idealc{Z}$ for some $Z
\subset \rbb^n$$)$.
   \end{enumerate}
   \end{lem}
   \begin{proof}
(i) is proved in \cite[pages 24 and 42]{c-s-sz}.

(ii) By Lemma \ref{cor}, there exists $Z \subset
\rbb^n$ such that $\idealf{\varLambda_\mathsf r} =
\ideal{Z}$. This in turn implies that
$\idealfc\varLambda = \idealf{\varLambda_\mathsf r} +
\I\, \idealf{\varLambda_\mathsf r} = \ideal{Z} + \I \,
\ideal{Z} = \idealc{Z}$.
   \end{proof}
   \begin{lem} \label{8.10.2007}
If $I$ is a $*$-ideal of $\pcal_n$ and $\varLambda$ is
a positive definite complex linear functional on
$\pcal_n$, then $\varLambda$ vanishes on $I$ if and
only if it vanishes on $\idealc{\zera{I}}$, where
$\zera I$ is defined in \eqref{zera}.
   \end{lem}
   \begin{proof}
The set $I_\mathsf r := \{P \in I \colon P^*=P\}$ is
an ideal of $\rcal_n$ and $I = I_\mathsf r + \I \,
I_\mathsf r$. Since $\zera{I} = \zera{I_\mathsf r}$
and $\idealc{Z} = \ideal{Z} + \I \, \ideal{Z}$ for
each subset $Z$ of $\rbb^n$, we get
   \begin{align}   \label{plus2}
\idealc{\zera{I}} = \ideal{\zera{I_\mathsf r}} + \I \,
\ideal{\zera{I_\mathsf r}}.
   \end{align}
Thus, we have
   \begin{align*}
\varLambda|_I = 0 \iff \varLambda_\mathsf
r|_{I_\mathsf r} = 0 \overset{\textrm{Lemma \ref{1}}}
\iff \varLambda_\mathsf r|_{\ideal{\zera{I_\mathsf
r}}} = 0 \overset{ \eqref{plus2} }\iff
\varLambda|_{\idealc{\zera{I}}} = 0,
   \end{align*}
   which completes the proof.
   \end{proof}
Take $P \in \pcal_n$. Following \cite[Section 10]{c-s-sz},
we say that $P\in \pcal_n$ is of {\em type} {\sf A}
(respectively {\sf B}) if each positive definite complex
linear functional $\varLambda$ on $\pcal_n$ vanishing on
$(P)_\mathsf{c}$ (respectively $\idealc{\mathcal Z_P}$) is
a moment functional. Clearly, each polynomial of type {\sf
A} is of type {\sf B}. By Lemma \ref{12.XII.06}, the
``complexified'' definitions of types {\sf A} and {\sf B}
coincide with the ``real'' ones for every $P \in \rcal_n$.
Applying again Lemma \ref{12.XII.06}, we see that $P \in
\pcal_n$ is of type {\sf A} (respectively {\sf B}) if and
only if $P^*P$ is of type {\sf A} (respectively {\sf B}).
This fact and the equality $\mathcal Z_S = \mathcal
Z_{S^*S}$, $S \in \pcal_n$, imply that Theorem \ref{popr}
remains valid for all $P,Q \in \pcal_n$.
   \begin{thm} \label{popr+}
Let $P,Q \in \pcal_n$ be such that $\mathcal Z_P = \mathcal
Z_Q$. Then $P$ is of type {\sf A} if and only if $Q$ is of
type {\sf A}.
   \end{thm}
Knowing this, we can define correctly type {\sf A} for an
algebraic subset of $\rbb^n$ induced by a polynomial $P \in
\pcal_n$. Finally, we see that Theorem \ref{A=B} remains
valid in the ``complexified'' setting as well.
   \begin{thm} \label{A=B+}
An algebraic subset of $\rbb^n$ induced by a polynomial $P
\in \pcal_n$ is of type {\sf A} if and only if it is of
type~ {\sf B}.
   \end{thm}
   \begin{cor} \label{3czl}
Let $V$ and $W$ be algebraic subsets of $\rbb^n$ such
that $V \subset W$. If $W$ is of type {\sf A}, then so
is $V$.
   \end{cor}
   \begin{proof}
Apply \cite[Proposition 48]{c-s-sz} and Theorem \ref{A=B+}.
   \end{proof}
   \section{\label{s4}Type C}
Following \cite[Proposition 42]{c-s-sz}, we say that a
$*$-ideal $I$ of $\pcal_n$ is of {\em type} {\sf C} if
each positive definite complex linear functional
$\varLambda$ on $\pcal_n$ satisfying the equality
$\idealfc\varLambda = I$ is a moment functional. An
algebraic subset $V$ of $\rbb^n$ is of {\em type} {\sf
C} if the ideal $\idealc{V}$ is of type {\sf C}. Each
algebraic set of type {\sf A} is automatically of type
{\sf C}. In Section \ref{s6} we will show that the
converse implication is true as well. In view of the
proof of Lemma \ref{11.XII.06}\,(ii), we see that an
algebraic subset $V$ of $\rbb^n$ is of type {\sf C} if
and only if each positive definite real linear
functional $\varLambda$ on $\rcal_n$ satisfying the
equality $\idealf\varLambda = \ideal V$ is a moment
functional.

We now complete \cite[Proposition 41]{c-s-sz}.
   \begin{pro}\label{wkw}
If $I$ is a proper $*$-ideal of $\pcal_n$, then the
following three conditions are equivalent
   \begin{enumerate}
   \item[(i)] $I$ is  a set ideal,
   \item[(ii)] there exists a positive definite
complex linear functional $\varLambda$ on $\pcal_n$
such that $\idealfc\varLambda = I$,
   \item[(iii)] there exists a rigid
$I$-basis $\{Q_k\}_{k=0}^m$ of $\pcal_n$ composed of
real column polynomials {\em (}with $Q_0=1${\em )} and
a complex linear functional $\varLambda$ on $\pcal_n$
such that $\varLambda(Q_kQ_l^*)$ is the zero matrix
for all $k \neq l$, $\varLambda(Q_kQ_k^*)$ is the
identity matrix for all $k$, and $I \subset \ker
\varLambda$.
   \end{enumerate}
   \end{pro}
   \begin{proof}
(i)$\Rightarrow$(ii) See the proof of
\cite[Proposition 41\,(ii)]{c-s-sz}.

(ii)$\Rightarrow$(i) Apply Lemma
\ref{11.XII.06}\,(ii).

(ii)$\Rightarrow$(iii) Since $\idealfc\varLambda = I$
and the ideal $I$ is proper, the functional
$\varLambda$ must be non-zero. Employing
\cite[Proposition 32\,(i)$\Rightarrow$(ii)]{c-s-sz},
\cite[Remark 33]{c-s-sz} and the fact that
$\idealfc\varLambda \subset \ker \varLambda$, we get
(iii).

(iii)$\Rightarrow$(ii) Use \cite[Theorem 36]{c-s-sz}.
   \end{proof}
In fact, in view of the proof of \cite[Proposition
41\,(ii)]{c-s-sz}, if $I$ is a set ideal of $\pcal_n$,
then there always exists a (complex linear) moment
functional $\varLambda$ on $\pcal_n$ such that
$\idealfc\varLambda = I$ (this is much more than
required in Proposition \ref{wkw}\,(ii)).
   \section{\label{s5}Semiperfect $*$-algebras}
Our next goal is to study a class of $*$-algebras in
which positive definite linear functionals are
automatically moment functionals. The results we
obtain will be used in Section \ref{s6}. In what
follows $\mathbb K$ stands either for $\rbb$ or
$\cbb$. In this section $A$ is assumed to be a
finitely generated commutative $*$-algebra over
$\mathbb K$ with a unit $e$ (the unit and the zero
elements of $A$ are always assumed to be different).
One can show that if $\mathbb K = \rbb$, then $A$ is
the direct sum of two vector spaces \mbox{$A_{\mathrm
R}:=\{a \in A \colon a^*=a\}$} and \mbox{$\{a \in A
\colon a^*=-a\}$}, the first of which is a finitely
generated unital $*$-algebra with the identity mapping
as involution. If $\mathbb K= \cbb$, then $A$ always
has a finite set of selfadjoint generators (if
necessary replace the old generators $a_1, \ldots,
a_n$ of $A$ by selfadjoint ones $\{a_k+a_k^*\}_{k=1}^n
\cup \{\mathrm i (a_k-a_k^*)\}_{k=1}^n$). Denote by
$\varDelta^*(A)$ the set of all multiplicative linear
functionals $\gamma \colon A \to \mathbb K$ such that
$\gamma(e)=1$ and $\gamma(a^*) = \overline{\gamma(a)}$
for every $a \in A$. Equip $\varDelta^*(A)$ with the
topology of pointwise convergence. For $a \in A$, we
define the continuous mapping $\hat a \colon
\varDelta^*(A) \to \mathbb K$ by $\hat a(\gamma) :=
\gamma(a)$ for $\gamma \in \varDelta^*(A)$. If ${\bf
a} := (a_1, \ldots, a_n)$ are generators of $A$ (which
are assumed to be selfadjoint in the case of $\mathbb
K = \cbb$), then the mapping
   \begin{align} \label{homeo}
\varPhi_{\bf a} \colon \varDelta^*(A) \ni \gamma \mapsto
(\gamma(a_1), \ldots, \gamma(a_n)) \in \mathcal
R(\varPhi_{\bf a}) \subset \rbb^n
   \end{align}
is a homeomorphic embedding of $\varDelta^*(A)$ onto a
closed subset $\mathcal R(\varPhi_{\bf a})$ of $\rbb^n$.
This implies that the $\sigma$-algebra of all Borel subsets
of $\varDelta^*(A)$ is the smallest $\sigma$-algebra of
subsets of $\varDelta^*(A)$ with respect to which all the
mappings $\hat a$, $a \in A$, are measurable. We say that
$A$ is {\em $\varDelta^*$-separative} if $\varDelta^*(A)
\neq \varnothing$ and $\bigcap_{\gamma \in \varDelta^*(A)}
\ker \gamma = \{0\}$ (equivalently:\ the mapping $A \ni a
\mapsto \hat a \in \mathbb K^{\varDelta^*(A)}$ is
injective). If $A$ is $\varDelta^*$-separative, then the
mapping $a \mapsto \hat a \circ \varPhi_{\bf a}^{-1}$ is a
$*$-algebra isomorphism between $*$-algebras $A$ and
$A_{\bf a}$, where $A_{\bf a}$ consists of all functions $f
\colon \mathcal R(\varPhi_{\bf a})\to \mathbb K$ for which
there exists $P \in \mathbb K[X_1, \ldots, X_n]$ such that
$f(x)=P(x)$ for all $x \in \mathcal R(\varPhi_{\bf a})$
($A_{\bf a}$ is equipped with pointwise defined algebraic
operations and involution $f^*(x)=\overline{f(x)}$ for $x
\in \mathcal R(\varPhi_{\bf a})$).

A linear functional $\varLambda \colon A \to \mathbb K$
such that $\varLambda(a^*a) \Ge 0$ for all $a \in A$ is
called {\em positive definite}. A linear functional
$\varLambda \colon A \to \mathbb K$ is said to be {\em
strictly positive definite} if $\varLambda(a^*a)
> 0$ for all $a \in A \setminus \{0\}$. Denote by
$\mathcal{PD}(A)$ the set of all positive definite linear
functionals on $A$ and equip it with the topology of
pointwise convergence. Clearly, $\varDelta^*(A)$ is a
closed subset of $\mathcal{PD}(A)$. We say that $A$ is {\em
semiperfect} if $\varDelta^*(A) \neq \varnothing$ and for
every $\varLambda \in \mathcal{PD}(A)$ there is a positive
Borel measure $\mu$ on $\varDelta^*(A)$ such that
   \begin{align}    \label{mfu}
\varLambda(a) = \int_{\varDelta^*(A)} \hat a \D \mu, \quad
a \in A.
   \end{align}
By \eqref{homeo}, such $\mu$ is automatically regular
(cf.\ \cite[Theorem 2.18]{rud}). The functional
$\varLambda$ of the form \eqref{mfu} is called a {\em
moment functional}. Applying the identification
\eqref{homeo} and the measure transport theorem, one
can deduce from the Riesz-Haviland characterization of
Hamburger moment problem on $\mathcal R(\varPhi_{\bf
a})$ (cf.\ \cite{rie,hav1,hav2}) that
   \begin{align} \label{hav}
\begin{minipage}{29em} if $\varDelta^*(A)\neq \varnothing$,
then a linear functional $\varLambda \colon A \to \mathbb
K$ is a moment functional if and only if $\varLambda (a)
\Ge 0$ for every $a \in A$ such that $\hat a \Ge 0$.
\end{minipage}
   \end{align}
It follows directly from \eqref{hav} that
   \begin{align} \label{hav-lim}
\begin{minipage}{29em} if $\varDelta^*(A)\neq \varnothing$
and $\varLambda_k \colon A \to \mathbb K$, $k \Ge 1 $, is a
sequence of moment functionals which is pointwise
convergent to a function $\varLambda \colon A \to \mathbb
K$, then $\varLambda$ is a moment functional.
   \end{minipage}
   \end{align}
   \begin{rem} \label{qotpd}
If $I$ is a proper $*$-ideal of $A$, then the quotient
$*$-algebra $A/I$ is finitely generated and the set
$\{\gamma \in \varDelta^*(A) \colon I \subset \ker
\gamma\}$ is a closed subset of $\varDelta^*(A)$ which is
homeomorphic to $\varDelta^*(A/I)$ via the mapping $\gamma
\mapsto \gamma_I$, where $\gamma_I(a + I) := \gamma(a)$ for
$a \in A$. Similarly, the set $\{\varLambda \in
\mathcal{PD}(A) \colon I \subset \ker \varLambda\}$ is a
closed subset of $\mathcal{PD}(A)$ which is homeomorphic to
$\mathcal{PD}(A/I)$ via the mapping $\varLambda \mapsto
\varLambda_I$, where $\varLambda_I(a + I) := \varLambda(a)$
for $a \in A$.
   \end{rem}
In what follows, we regard $\mathbb K[X_1, \ldots, X_n]$ as
the $*$-algebra over $\mathbb K$ with the unique involution
such that $X_j^*=X_j$ for all $j=1,\ldots,n$. Thus the
involution of $\rbb[X_1, \ldots,X_n]$ is the identity
mapping. Note that the $*$-algebra $A$ which has
selfadjoint generators is $*$-isomorphic to a quotient
$*$-algebra $\mathbb K[X_1, \ldots, X_n]/I$, where $I$ is a
proper $*$-ideal of $\mathbb K[X_1, \ldots, X_n]$ (apply
the homomorphism theorem to the $*$-algebra epimorphism
$\mathbb K[X_1, \ldots, X_n] \ni P \mapsto P(a_1, \ldots,
a_n) \in A$, where $a_1, \ldots, a_n$ are selfadjoint
generators of $A$). Otherwise, $\mathbb K = \rbb$ and
$A_{\mathrm R} \neq A$. In this particular case, the
$*$-algebra $A$ is $*$-isomorphic to a quotient $*$-algebra
$\rbb[X_1, \ldots, X_n, Y_1, \ldots, Y_n]/I$, where $I$ is
a proper $*$-ideal of the $*$-algebra $\rbb[X_1, \ldots,
X_n, Y_1, \ldots, Y_n]$ which is equipped with the unique
involution such that $X_j^*=Y_j$ for all $j=1,\ldots,n$.

We now discuss $\varDelta^*$-separativity of the model
$*$-algebra $\mathbb K[X_1, \ldots, X_n]/I$ leaving the
case $\rbb[X_1, \ldots, X_n, Y_1, \ldots, Y_n]/I$ aside (it
is more complicated and not essential for the main purpose
of this paper).
   \begin{lem} \label{quo} Let $I$ be a proper $*$-ideal of
$A = \mathbb K[X_1, \ldots, X_n]$. Then the following
assertions hold:
   \begin{enumerate}
   \item[(i)] $\varDelta^*(A/I) \neq \varnothing$ if and only
if $\zera I \neq \varnothing$, where $\zera I$ is defined
in \eqref{zera},
   \item[(ii)]  if $\zera{I} \neq \varnothing$, then the
mapping
   \begin{align*}
\varPsi_I \colon \varDelta^* (A/I) \ni \chi \mapsto
(\chi(X_1 + I), \ldots, \chi(X_n + I)) \in \zera I
   \end{align*}
is a well defined homeomorphism such that
   \begin{align}     \label{chiP+I}
\chi(P+I) = P(\varPsi_I(\chi)), \quad P \in A, \, \chi \in
\varDelta^* (A/I),
   \end{align}
   \item[(iii)] $A/I$ is $\varDelta^*$-separative if and
only if $I$ is a set ideal.
   \end{enumerate}
   \end{lem}
   \begin{proof}
(i) $\&$ (ii) Suppose that $\varDelta^* (A/I) \neq
\varnothing$. It is easily seen that \eqref{chiP+I}
holds. Moreover, $\varPsi_I(\chi) \in \zera I$ for every
$\chi \in \varDelta^* (A/I)$, because
   \begin{align*}
P(\varPsi_I(\chi)) \overset{\eqref{chiP+I}} = \chi(P+I) =
\chi(0+I) = 0, \quad P \in I.
   \end{align*}
Conversely, if $x \in \zera I$, then the mapping
$\gamma_x \colon A/I \ni P+I \mapsto P(x) \in \mathbb K$
is well defined and $\gamma_x \in \varDelta^*(A/I)$.
What is more, $\varPsi_I(\gamma_x) = x$. It follows from
\eqref{chiP+I} that $\gamma_{\varPsi_I(\chi)} = \chi$
for all $\chi \in \varDelta^*(A/I)$, which means that
$\varPsi_I^{-1}(x) = \gamma _x$ for all $x \in \zera I$.
That $\varPsi_I$ is a homeomorphism is now obvious.

(iii) Suppose that $I = \ideal Z :=\{P \in A \colon
\forall \, x \in Z \;\; P(x)=0\}$, where $Z \subset
\rbb^n$. If $P \in A \setminus I$, then there exists
$x\in Z \subset \zera I$ such that $P(x) \neq 0$.
Hence $\gamma_x(P+I) \neq 0$, which means that $A/I$
is $\varDelta^*$-separative. Conversely, if $A/I$ is
$\varDelta^*$-separative, then $\bigcap_{\chi \in
\varDelta^*(A/I)} \ker \chi = \{0\}$. In view of (i)
and (ii), this equality implies that every $P \in A$
which vanishes on $\zera I (\neq \varnothing)$ belongs
to $I$. In other words, $\ideal{\zera I} \subset I$.
Since the reverse inclusion is always true, we get
$I=\ideal{\zera I}$.
   \end{proof}
Referring to Lemma \ref{quo}, note that if $I$ is a set
ideal of $\mathbb K[X_1, \ldots, X_n]$ of the form $I =
\ideal{Z}$, where $Z \subset \rbb^n$, then $\zera{I}$ is
the closure of $Z$ in the Zariski topology, and
$I=\ideal{\zera{I}}$ (use Lemma \ref{dodat} and consult
\cite[Section 9]{c-s-sz}).

We show that if the $*$-algebra $A$ has selfadjoint
generators and $\varDelta^*(A)=\varnothing$, then
$\mathcal{PD}(A) = \{0\}$, which means that $A$ can be
thought of as a semiperfect $*$-algebra.
   \begin{pro}\label{pd0}
Let $A$ be a finitely generated commutative unital
$*$-algebra over $\mathbb K$, which in the case of $\mathbb
K=\rbb$ is equipped with the identity involution. Then
$\varDelta^*(A)=\varnothing$ if and only if
$\mathcal{PD}(A) = \{0\}$.
   \end{pro}
   \begin{proof}
The ``if'' part of the conclusion is obvious. Suppose now
that $\varDelta^*(A)=\varnothing$. As we know, the
$*$-algebra $A$ is $*$-isomorphic to a quotient $*$-algebra
$B/I$, where $I$ is a proper $*$-ideal of $B=\mathbb K[X_1,
\ldots, X_n]$. By Lemma \ref{quo}\,(i), we have $\zera I =
\varnothing$. This implies that $\ideal{\zera I}=B$. Take
$\varXi \in \mathcal{PD}(B/I)$ and set $\varLambda(Q) =
\varXi(Q+I)$ for $Q \in B$. Then $\varLambda \in
\mathcal{PD}(B)$ and $\varLambda|_I=0$. This, together with
Lemma \ref{1} (the real case) and Lemma \ref{8.10.2007}
(the complex case), implies that the functional
$\varLambda$ vanishes on $\ideal{\zera I}=B$. Hence,
$\varXi = 0$, which completes the proof.
   \end{proof}
Below, we characterize semiperfectness of the model
$*$-algebra $\mathbb K[X_1, \ldots, X_n]/I$.
   \begin{pro}  \label{tabsem}
If $I$ is a $*$-ideal of $A=\mathbb K[X_1, \ldots, X_n]$,
then the following two conditions are equivalent{\em :}
   \begin{enumerate}
   \item[(i)] $\zera I \neq \varnothing$ and every positive
definite linear functional
$\varLambda \colon A \to \mathbb K$ vanishing on $I$ is a
moment functional on $\zera I$,
   \item[(ii)] the ideal $I$ is
proper and the $*$-algebra $A/I$ is semiperfect.
   \end{enumerate}
   \end{pro}
   \begin{proof}
   (i)$\Rightarrow$(ii) As $\zera I \neq \varnothing$, we
see that $I \neq A$ and, by Proposition \ref{quo}\,(i),
$\varDelta^*(A/I) \neq \varnothing$. Take $\varXi \in
\mathcal{PD}(A/I)$ and set $\varLambda(Q) = \varXi(Q+I)$
for $Q \in A$. Then $\varLambda \in \mathcal{PD}(A)$ and
$\varLambda|_I=0$. Hence $\varLambda$ is a moment
functional on $\zera{I}$ with a representing measure $\mu$.
By Lemma \ref{quo}\,(ii) and the measure transport theorem,
we have
   \allowdisplaybreaks
   \begin{align*}
\varXi(Q+I) = \int_{\zera{I}} Q(x) \text{ d} \mu(x) &
\hspace{.4em} = \int_{\varDelta^* (A/I)} Q(\varPsi_I
(\chi)) \text{ d} (\mu \circ \varPsi_I) (\chi)
   \\
   & \overset{\eqref{chiP+I}}= \int_{\varDelta^* (A/I)}
\chi(Q+I) \text{ d} (\mu \circ \varPsi_I) (\chi), \quad Q
\in A,
   \end{align*}
where $\mu \circ \varPsi_I(\sigma)=\mu(\varPsi_I(\sigma))$
for a Borel subset $\sigma$ of $\varDelta^* (A/I)$. This
gives (ii).

   (ii)$\Rightarrow$(i) By Proposition \ref{quo}\,(i), we
have $\zera I \neq \varnothing$. If $\varLambda \in
\mathcal{PD}(A)$ vanishes on $I$, then the mapping $\varXi
\colon A/I \ni Q+I \mapsto \varLambda(Q) \in \mathbb K$ is
well defined and $\varXi \in \mathcal{PD}(A/I)$. Hence
there exists a positive Borel measure $\nu$ on $\varDelta^*
(A/I)$ such that
   \allowdisplaybreaks
   \begin{align*}
\varLambda(Q) = \int_{\varDelta^* (A/I)} \chi(Q+I) \text{
d} \nu (\chi) & \overset{\eqref{chiP+I}}= \int_{\varDelta^*
(A/I)} Q(\varPsi_I (\chi)) \text{ d} \nu (\chi)
   \\
   & \hspace{.41em}= \int_{\zera{I}} Q(x) \text{ d} (\nu
\circ \varPsi_I^{-1}) (x), \quad Q \in A,
   \end{align*}
which completes the proof.
   \end{proof}
   \begin{cor} \label{a/i=a/j}
Let $I$ and $J$ be proper $*$-ideals of $A=\mathbb
K[X_1, \ldots, X_n]$ such that $\zera I = \zera J$ and
for every $\varLambda \in \mathcal{PD}(A)$,
$\varLambda|_I=0$ if and only if $\varLambda|_J=0$.
Then the $*$-algebra $A/I$ is semiperfect if and only
if the $*$-algebra $A/J$ is semiperfect.
   \end{cor}
Recall that $(P)$ is the ideal of $\mathbb K[X_1,
\ldots, X_n]$ generated by a polynomial $P \in \mathbb
K[X_1, \ldots, X_n]$ and $\mathcal Z_P := \{x \in
\rbb^n \colon P(x)=0\}$.
   \begin{cor} \label{coralg}
If $P \in A=\mathbb K[X_1, \ldots, X_n]$, $P=P^*$ and
$\mathcal Z_P \neq \varnothing$, then the ideals $(P)$ and
$\ideal{\mathcal Z_P}$ are proper and the following two
conditions are equivalent\/{\em :}
   \begin{enumerate}
   \item[(i)] the polynomial $P$ is of type {\sf A}
$($respectively {\sf B}$)$,
   \item[(ii)] the $*$-algebra $A/(P)$
$($respectively $A/\ideal{\mathcal Z_P}$$)$ is semiperfect.
   \end{enumerate}
   \end{cor}
   \begin{proof} Since $\zera{(P)} = \mathcal Z_P$ and
$\zera{\ideal{\mathcal Z_P}}=\mathcal Z_P$, we can
apply Proposition \ref{tabsem} to $*$-ideals $(P)$ and
$\ideal{\mathcal Z_P}$ (use also \cite[Proposition
2.1]{sto}).
   \end{proof}
It may happen that $\mathbb K[X_1, \ldots, X_n]/I$ is
semiperfect but not $\varDelta^*$-separative.
   \begin{exa} \label{sep}
The proper $*$-ideal $I_1:=(1+X_1^2 + X_2^2)$ of $A=\mathbb
K[X_1, X_2]$ has the property that $\zera {I_1} =
\varnothing$, which by Lemma \ref{quo}\,(i) means that
$\varDelta^*(A/{I_1}) = \varnothing$. In turn, if the
polynomial $P\in A$ is of the form $P(X_1,X_2):=X_1^2$,
then the $*$-ideal $I_2:= (P)$ of $A$ is proper and $\zera
{I_2} \neq \varnothing$. By Lemma \ref{quo}\,(i), this
implies that $\varDelta^*(A/I_2) \neq \varnothing$. Since
$I_2$ is not a set ideal, the $*$-algebra $A/I_2$ is not
$\varDelta^*$-separative (cf.\ Lemma \ref{quo}\,(iii)).
However, by \cite[Theorem 5.4]{sto}, the polynomial $P$ is
of type {\sf A} (as a member of $A$) and hence, by
Corollary \ref{coralg}, the $*$-algebra $A/I_2$ is
semiperfect.
   \end{exa}
Recall that if $A=\mathbb K[X_1, \ldots, X_n]$ and
$\varLambda \in \mathcal{PD}(A)$, then the set
   \begin{align*}
\idealf \varLambda := \{P \in A \colon \varLambda(P^*P)=0\}
   \end{align*}
is the greatest ideal of $A$ contained in the kernel of
$\varLambda$ (see Lemmata \ref{cor} and \ref{11.XII.06}).
   \begin{pro}  \label{tabsem+}
Let $I$ be a $*$-ideal of $A=\mathbb K[X_1, \ldots, X_n]$
such that $\zera I \neq \varnothing$. Then the ideal $I$ is
proper and the following two conditions are equivalent{\em
:}
   \begin{enumerate}
   \item[(i)] every positive definite linear functional
$\varLambda \colon A \to \mathbb K$ such that $\idealf
\varLambda=I$ is a moment functional on $\zera I$,
   \item[(ii)] every strictly  positive definite linear functional
$\varXi \colon A/I \to \mathbb K$ is a moment functional.
   \end{enumerate}
   \end{pro}
   \begin{proof}
Applying Remark \ref{qotpd}, we verify that the mapping
$\varLambda \mapsto \varLambda_I$ is a bijection between
the set of all positive definite linear functionals
$\varLambda \colon A \to \mathbb K$ such that $\idealf
\varLambda=I$ and the set of all strictly positive definite
linear functionals $\varXi \colon A/I \to \mathbb K$. This
enables us to repeat arguments used in the proof of
Proposition \ref{tabsem}.
   \end{proof}
   \begin{exa} \label{kontrprz}
Consider the $*$-algebra $A = \cbb[X_1,X_2]$ and the proper
$*$-ideal $I$ of $A$ generated by the polynomial $P := (X_2
- X_1^2)X_2 \in A$. By \cite[Theorem 6.3]{sto}, the
polynomial $P$ is not of type {\sf A}. This means that
there exists at least one positive definite complex linear
functional on $A$ vanishing on $I$, which is not a moment
functional. Let $\varLambda$ be any such functional. By
Lemma \ref{11.XII.06}\,(ii), the ideal $\idealf\varLambda$
is a set ideal. This fact, when combined with the
discussion in \cite[Example 54]{c-s-sz}, implies that
$\idealf\varLambda = \ideal{\mathcal Z_P} = I$. Thus, the
mapping $\varLambda_I\colon A/I \ni Q + I \mapsto
\varLambda(Q) \in \cbb$ is a well defined strictly positive
definite linear functional on $A/I$ which is not a moment
functional. Summarizing, we have shown that the $*$-algebra
$A/I$ is not semiperfect and each positive definite linear
functional on $A/I$ which is not a moment functional is
automatically strictly positive definite.
   \end{exa}
   \begin{lem} \label{skgenid}
Let $I$ be a $*$-ideal of $A=\mathbb K[X_1, \ldots,
X_n]$ generated by a finite set of polynomials $P_1,
\ldots, P_s \in A$. Set $P=P_1^*P_1 + \ldots +
P_s^*P_s$. Then
   \begin{enumerate}
   \item[(i)] for every $\varLambda \in \mathcal {PD}(A)$,
$\varLambda|_I=0$ if and only if
$\varLambda|_{(P)}=0$,
   \item[(ii)] if $I \neq A$ $($hence $(P) \neq A$$)$, then the
$*$-algebra $A/I$ is semiperfect if and only if the
$*$-algebra $A/(P)$ is semiperfect.
   \end{enumerate}
   \end{lem}
   \begin{proof}
(i) Take $\varLambda \in \mathcal {PD}(A)$. If
$\varLambda|_{(P)}=0$, then $0= \varLambda(P_1^*P_1) +
\ldots + \varLambda(P_s^*P_s)$, which implies that $P_1,
\ldots, P_s \in \idealf \varLambda$. Since $\idealf
\varLambda$ is a $*$-ideal of $A$ contained in $\ker
\varLambda$, we deduce that $I \subset \idealf \varLambda
\subset \ker \varLambda$. That $\varLambda|_I=0$ implies
$\varLambda|_{(P)}=0$ follows from $(P) \subset I$.

(ii) If $I \neq A$, then the ideal $J:=(P)$ is proper and
   \begin{align*}
\zera{I} = \mathcal Z_{P_1} \cap \ldots \cap \mathcal
Z_{P_s} = \mathcal Z_P = \zera J,
   \end{align*}
which, when combined with (i) and Corollary \ref{a/i=a/j},
completes the proof.
   \end{proof}
   We now show that the question of semiperfectness of the
model $*$-algebra $\mathbb K[X_1, \ldots, X_n]/I$ can be
reduced to studying polynomials of type {\sf A}.
   \begin{pro}\label{reduction}
If $I$ is a proper $*$-ideal of $A=\mathbb K[X_1, \ldots,
X_n]$, then there exists $P \in I$ such that $P=P^*$,
$(P)\neq A$ and the following three conditions are
equivalent\/{\em :}
   \begin{enumerate}
   \item[(i)] $A/I$ is semiperfect,
   \item[(ii)] $A/(P)$ is semiperfect,
   \item[(iii)] $\mathcal Z_P \neq \varnothing$ and the polynomial $P$ is of type {\sf A}.
   \end{enumerate}
   \end{pro}
   \begin{proof}
Apply Hilbert's basis theorem, Lemma \ref{skgenid},
Corollary \ref{coralg} and Lemma \ref{quo}\,(i).
   \end{proof}
   \section{\label{s6}Types A and C coincide}
In this section we prove that every algebraic set of type
{\sf C} is of type {\sf A}. We begin with some auxiliary
results. Denote by $\card A$ the cardinal number of a set~
$A$.
   \begin{lem} \label{b1}
Let $X$ be a vector space over $\mathbb K$ and $\mathcal F$
be a family of vector subspaces of $X$. If $\mathcal F$
contains a vector space $E_0$ with $\dim E_0 \Le \aleph_0$,
then there exists $\mathcal F_0 \subset \mathcal F$ such
that $1 \Le \card {\mathcal F_0} \Le 1 + \dim E_0$ and
$\bigcap \mathcal F = \bigcap \mathcal F_0$.
   \end{lem}
   \begin{proof}
Consider first the case $\dim E_0 < \aleph_0$. We have two
possibilities: either $E_0 \cap E = E_0$ for every $E \in
\mathcal F$ and so $\bigcap \mathcal F = E_0$, or there
exists $E_1 \in \mathcal F$ such that $E_0 \cap E_1
\subsetneq E_0$. Continuing this procedure, we successively
find $E_0, E_1, \ldots, E_k \in \mathcal F$ ($k \Ge 1$)
such that $E_0 \cap \ldots \cap E_k \subsetneq E_0 \cap
\ldots \cap E_{k-1} \subsetneq \ldots \subsetneq E_0$.
Since always $k \Le \dim E_0$, the procedure must stop at
$k=n \Le \dim E_0$. Then $E_0 \cap \ldots \cap E_n \cap E=
E_0 \cap \ldots \cap E_n$ for every $E \in \mathcal F$,
which implies that $\bigcap \mathcal F = E_0 \cap \ldots
\cap E_n$.

Suppose now that $\dim E_0 = \aleph_0$. Then there exist
vector subspaces $\{F_n\}_{n=1}^\infty$ of $E_0$ such that
$\dim F_n = n$ for every $n \Ge 1$, and $E_0 =
\bigcup_{n=1}^\infty F_n$. Set $\mathcal F_n = \{E \cap F_n
\colon E \in \mathcal F\}$ for $n \Ge 1$. By the first part
of the proof, for every $n\Ge 1$ there exists $\mathcal
F_{n,0} \subset \mathcal F$ such that $1 \Le \card
{\mathcal F_{n,0}} \Le 1+n$ and $\bigcap \mathcal F_n =
\bigcap \{E \cap F_n \colon E \in \mathcal F_{n,0}\}$.
Then, as easily seen, $\mathcal F_0 := \{E_0\} \cup
\bigcup_{n=1}^\infty \mathcal F_{n,0}$ is the required
family.
   \end{proof}
   \begin{lem} \label{b2}
Let $A$ be a $\varDelta^*$-separative finitely generated
commutative $*$-algebra over $\mathbb K$ with unit. Then
there exists a moment functional $\varLambda \colon A \to
\mathbb K$ which is strictly positive definite.
   \end{lem}
   \begin{proof}
Since $\dim A \Le \aleph_0$ and $A$ is
$\varDelta^*$-separative, Lemma \ref{b1} implies that there
exists a sequence $\{\gamma_j\}_{j=1}^\varkappa \subset
\varDelta^*(A)$ with $\varkappa \in \{1,2, \ldots\} \cup
\{\infty\}$ such that
   \begin{align} \label{fin}
\bigcap_{j=1}^\varkappa \ker \gamma_j = \{0\}.
   \end{align}
If $\varkappa < \infty$, then, by \eqref{fin},
$\varLambda(a) = \sum_{j=1}^\varkappa \gamma_j(a)$, $a \in
A$, defines the required functional. If $\varkappa =
\infty$, then $\varLambda$ can be defined by $\varLambda(a)
= \sum_{j=1}^\infty c_j \gamma_j(a)$ for $a \in A$, where
$c_j = j^{-2}\exp(-\delta_j)$, $\delta_j = \max_{k=1}^n
|\gamma_j(a_k)|$ for $j=1, 2, \ldots$, and $a_1, \ldots,
a_n$ are generators of $A$. Since $|t|^k \Le k! \exp(|t|)$
for all $t\in \rbb$ and $k=0,1,2, \ldots$, the series
defining $\varLambda$ is absolutely convergent. The
representing measure $\mu$ of $\varLambda$ is given by
$\mu(\sigma) = \sum_{j=1}^\infty c_j \chi_\sigma(\gamma_j)$
for a Borel subset $\sigma$ of $\varDelta^*(A)$, where
$\chi_\sigma$ is the characteristic function of $\sigma$.
That $\varLambda$ is strictly positive definite follows
from \eqref{fin}.
   \end{proof}
   \begin{lem} \label{pdstr}
Let $A$ be a $\varDelta^*$-separative finitely generated
commutative $*$-algebra over $\mathbb K$ with unit. If
every strictly positive definite linear functional on $A$
is a moment functional, then the $*$-algebra $A$ is
semiperfect.
   \end{lem}
   \begin{proof}
By Lemma \ref{b2}, there exists a moment functional
$\varTheta \colon A \to \mathbb K$ which is strictly
positive definite. Take $\varLambda \in \mathcal{PD}(A)$
and set $\varLambda_k = \varLambda + \frac 1k \varTheta$
for every integer $k \Ge 1$. Then
$\{\varLambda_k\}_{k=1}^\infty$ is a sequence of strictly
positive definite linear functionals on $A$ which pointwise
converges to $\varLambda$. By \eqref{hav-lim} and our
assumption, the functional $\varLambda$ is a moment
functional, which completes the proof.
   \end{proof}
   Now we are in a position to formulate the main result of
the paper.
   \begin{thm} \label{typC=A}
An algebraic subset of $\rbb^n$ is of type {\sf A} if and
only if it is of type~ {\sf C}.
   \end{thm}
   \begin{proof}
Let $V$ be an algebraic subset of $\rbb^n$ of type {\sf C}.
Then $V=\mathcal Z_P$ for some $P \in A=\cbb[X_1, \ldots,
X_n]$ such that $P=P^*$. If $V=\varnothing$, then
$V=\mathcal Z_Q$, where $Q \in A$ is given by $Q(x)=1$ for
$x \in \rbb^n$. Clearly, the polynomial $Q$ is of type {\sf
A}. Hence, by Theorem \ref{popr+}, the algebraic set $V$ is
of type {\sf A}. Consider now the case $\mathcal Z_P\neq
\varnothing$. Applying \cite[Proposition 2.1]{sto} and
Proposition \ref{tabsem+} to the $*$-ideal
$I=\ideal{\mathcal Z_P}$ (recall that
$\zera{\ideal{\mathcal Z_P}}=\mathcal Z_P$), we deduce that
every strictly positive definite linear functional $\varXi
\colon A/\ideal{\mathcal Z_P} \to \mathbb K$ is a moment
functional. Hence, by Lemma \ref{quo}\,(iii) and Lemma
\ref{pdstr}, the $*$-algebra $A/\ideal{\mathcal Z_P}$ is
semiperfect. This and Corollary \ref{coralg} imply that the
algebraic set $V$ is of type {\sf B}. In view of Theorem
\ref{A=B+}, $V$ is of type {\sf A}. Since every algebraic
set of type {\sf A} is of type {\sf C}, the proof is
complete.
      \end{proof}
Theorem \ref{typC=A} answers in the affirmative
Question 2 posed in \cite{c-s-sz}, which asks whether
types {\sf A}, {\sf B}, and {\sf C} coincide.
Combining Theorem \ref{typC=A} with any example of an
algebraic subset $V$ of $\rbb^n$ which is not of type
{\sf A} and for which $\idealc{V} \neq \{0\}$ (see
e.g.\ \cite{bis,sto}), we deduce that there exist
non-zero set ideals of $\cbb[X_1, \ldots, X_n]$ which
are not of type {\sf C} ($n$ must be greater than or
equal to $2$). This answers in the negative Question 3
posed in \cite{c-s-sz}. Employing again Theorem
\ref{typC=A} and the fact that for every $n \Ge 2$,
the zero polynomial in $\cbb[X_1, \ldots, X_n]$ is not
of type {\sf A} (cf.\ \cite[statement (50)]{c-s-sz}),
we see that for every integer $n \Ge 2$, the zero
ideal of $\cbb[X_1, \ldots, X_n]$ is not of type {\sf
C}. This answers in the negative Question 4 posed in
\cite{c-s-sz} (see also \cite{fr} for the case $n=2$).
Note, by the way, that every $*$-ideal of $\cbb[X_1]$
is of type {\sf C} (see \cite[Section 12]{c-s-sz}).
   \section{\label{s7}The multidimensional complex moment problem}
In what follows, we regard $\pcal_{n,n}:=\cbb[X_1, \ldots,
X_n, Y_1, \ldots, Y_ n]$ as the complex $*$-algebra with
the unique involution such that $X_j^*=Y_j$ for all $j=1,
\ldots, n$. Recall that $\pcal_{2n}$ stands for the complex
$*$-algebra $\cbb[X_1, \ldots, X_{2n}]$ equipped with the
unique involution such that $X_j^*=X_j$ for all $j=1,
\ldots, 2n$ (cf.\ Section \ref{s3}). Using the
homeomorphism $\cbb^n \ni z \mapsto \gamma_z \in
\varDelta^*(\pcal_{n,n})$ given by
   \begin{align*}
\gamma_z(Q)=Q(z,\bar z), \quad z \in \cbb^n, \, Q\in
\pcal_{n,n},
   \end{align*}
and the measure transport theorem, we see that a map
$\varLambda \colon \pcal_{n,n} \to \cbb$ is a moment
functional if and only if there exists a positive Borel
measure $\mu$ on $\cbb^n$ such that
   \begin{align}    \label{commpr}
\varLambda(Q) = \int_{\cbb^n} Q(z,\bar z) \text{ d} \mu(z),
\quad Q\in \pcal_{n,n}.
   \end{align}
This is exactly the $n$-dimensional complex moment
problem, which entails determining whether for a given
map $\varLambda \colon \pcal_{n,n} \to \cbb$ there
exists a positive Borel measure $\mu$ on $\cbb^n$
satisfying condition \eqref{commpr}. The interested
reader is referred to
\cite{mgkr,dev,kil,atz,sza1,sza2,fug,b-ch-r,bis0,put,c-p,st-sz,st-sz2,st-sz3,fial}
for further information on this subject. We now adapt
the notions of types {\sf A}, {\sf B}, and {\sf C} to
the context of the complex moment problem. First, note
that a real algebraic subset $V$ of $\cbb^n=\rbb^{2n}$
can be written in terms of the complex variable as
follows
   \begin{align*}
V=\mathcal Z_P:= \{z\in \cbb^n \colon P(z,\bar z)=0\} \quad
\text{for some } P \in \pcal_{n,n}.
   \end{align*}
If \eqref{commpr} holds and $P \in \pcal_{n,n}$, then the
closed support of the measure $\mu$ is contained in
$\mathcal Z_P$ if and only if the functional $\varLambda$
vanishes on the ideal $(P)$ of $\pcal_{n,n}$ generated by
$P$, or equivalently on the ideal $\ideal{\mathcal Z_P} :=
\{Q\in \pcal_{n,n} \colon Q|_{\mathcal Z_P}=0\}$ (cf.\
\cite[Proposition 1]{st-sz}). We say that a polynomial $P
\in \pcal_{n,n}$, or the real algebraic set $V=\mathcal
Z_P$, is of {\em type} {\sf A} (respectively:\ {\sf B},
{\sf C}) with respect to $\pcal_{n,n}$ if every positive
definite linear functional $\varLambda \colon \pcal_{n,n}
\to \cbb$ for which $(P) \subset \ker \varLambda$
(respectively:\ $\ideal{V} \subset \ker \varLambda$,
$\ideal{V}=\idealf{\varLambda}$) is a moment functional;
here $\idealf{\varLambda}:= \{Q \in \pcal_{n,n}\colon
\varLambda(Q^*Q)=0\}$. Note now that the mapping $\varGamma
\colon \pcal_{2n} \to \pcal_{n,n}$ defined by
   \begin{align*}
P \mapsto P\Big(\frac 12 (X_1+Y_1), \ldots, \frac 12
(X_n+Y_n), \frac 1{2\,\I} (X_1-Y_1), \ldots, \frac 1{2\,\I}
(X_n-Y_n)\Big)
   \end{align*}
is a $*$-algebra isomorphism between $*$-algebras
$\pcal_{2n}$ and $\pcal_{n,n}$. It is easily seen that
$\varGamma((P)) =(\varGamma(P))$, $\mathcal
Z_{\varGamma(P)} = \phi^{-1}(\mathcal Z_P)$ and
$\varGamma(\ideal{\mathcal Z_P}) =\ideal{\mathcal
Z_{\varGamma(P)}}$ for every $P \in \pcal_{2n}$, where
$\phi$ is the mapping from $\cbb^n$ to $\rbb^{2n}$ given by
   \begin{align*}
\phi(z_1, \ldots, z_n) = (\mathfrak{Re} \, z_1 , \ldots,
\mathfrak{Re} \, z_n, \mathfrak{Im} \, z_1, \ldots,
\mathfrak{Im} \, z_n), \quad z_1, \ldots, z_n \in \cbb.
   \end{align*}
This enables us to show that a polynomial $P \in
\pcal_{n,n}$ is of type {\sf A} (respectively:\ {\sf B},
{\sf C}) with respect to $\pcal_{n,n}$ if and only if the
polynomial $\varGamma^{-1}(P) \in \pcal_{2n}$ is of type
{\sf A} (respectively:\ {\sf B}, {\sf C}) with respect to
$\pcal_{2n}$. Hence, we can immediately adapt Theorems
\ref{popr+} and \ref{typC=A} to the context of the
$n$-dimensional complex moment problem.
   \begin{thm} \label{popr+c}
Let $P,Q \in \pcal_{n,n}$ be such that $\mathcal Z_P =
\mathcal Z_Q$. Then $P$ is of type {\sf A} with respect to
$\pcal_{n,n}$ if and only if $Q$ is of type {\sf A} with
respect to $\pcal_{n,n}$.
   \end{thm}
   \begin{thm} \label{typC=Ac}
A real algebraic subset of $\cbb^n=\rbb^{2n}$ is of type
{\sf A} with respect to $\pcal_{n,n}$ if and only if it is
of type {\sf C} with respect to $\pcal_{n,n}$.
   \end{thm}
   \appendix
   \section{\label{s8}}
Here, we aim to give the complex counterparts of
Lemmata \ref{realradical} and \ref{rn}, and the
equality (ii) of Lemma \ref{cor}. We preserve the
notation of Section \ref{s3}. In particular, $\rcal_n$
and $\pcal_n$ are real and complex polynomial rings
which are regarded as $*$-algebras ($\rcal_n$ with the
identity involution).

A $*$-ideal $I$ of $\pcal_n$ is said to be {\em
$*$-real} if for every finite sequence $Q_1, \ldots,
Q_s \in \pcal_n$ such that $Q_1^*Q_1 + \dots +
Q_s^*Q_s \in I$, we have $Q_1, \ldots, Q_s \in I$.
Recall that if $I$ is a $*$-ideal of $\pcal_n$, then
$I_\mathsf r := \{P \in I \colon P^*=P\}$ is an ideal
of $\rcal_n$ and $I = I_\mathsf r + \I \, I_\mathsf
r$. It is a routine matter to show that a $*$-ideal
$I$ of $\pcal_n$ is $*$-real if and only if the ideal
$I_\mathsf r$ of $\rcal_n$ is real. Moreover, if $J$
is a real ideal of $\rcal_n$, then $J + \I J$ is a
$*$-ideal of $\pcal_n$ which is $*$-real. Given a
$*$-ideal $I$ of $\pcal_n$, we denote by
$\sqrt[\mathrm{*R}]I$ the smallest $*$-real $*$-ideal
of $\pcal_n$ containing $I$, and call it the {\em
$*$-real radical} of $I$. In view of the above
discussion, the following can be deduced from Lemma
\ref{realradical}.
   \begin{lem} \label{*rad}
If $I$ is a $*$-ideal of $\pcal_n$, then
$\sqrt[\mathrm{*R}]I = \sqrt[\mathrm{R}]{I_\mathsf r} + \I
\sqrt[\mathrm{R}]{I_\mathsf r}$.
   \end{lem}
   \begin{thm}[$*$-Real Nullstellensatz] \label{2}
If $I$ is a $*$-ideal of $\pcal_n$, then
   \begin{align*}
\idealc{\zera{I}} = \sqrt[\mathrm{*R}]I.
   \end{align*}
   \end{thm}
   \begin{proof}
Apply Lemma \ref{*rad}, Lemma \ref{rn} and the equality
 \eqref{plus2}.
   \end{proof}
   \begin{pro}\label{opis*r}
If $I$ is a $*$-ideal of $\pcal_n$ and $P \in
\pcal_n$, then $P$ belongs to $\sqrt[\mathrm{*R}]I$ if
and only if there exist finitely many polynomials
$Q_1, \ldots, Q_s \in \pcal_n$ and an integer $m \Ge
0$ such that $(P^*P)^{2m} + Q_1^*Q_1 + \dots +
Q_s^*Q_s \in I$.
   \end{pro}
   \begin{proof}
This can be deuced from Lemma \ref{*rad} and Lemma
\ref{realradical} via standard algebraic arguments.
   \end{proof}
We say that a proper $*$-ideal of $\pcal_n$ is {\em
$*$-prime} if for all $P,Q \in \pcal_n$ such that $PQ
\in I$ and $PQ^* \in I$, either $P$ or $Q$ belongs to
$I$. It is a matter of routine to verify that a
$*$-ideal $I$ of $\pcal_n$ is $*$-prime if and only if
the ideal $I_\mathsf r$ of $\rcal_n$ is prime.
Moreover, if $J$ is a prime ideal of $\rcal_n$, then
$J + \I J$ is a $*$-ideal of $\pcal_n$ which is
$*$-prime. This, combined with Lemma \ref{*rad} and
the definition of $\sqrt[\mathrm{R}]{I_\mathsf r}$,
leads to the following description of
$\sqrt[\mathrm{*R}]I$.
   \begin{pro}\label{*opr}
If $I$ is a $*$-ideal of $\pcal_n$, then
$\sqrt[\mathrm{*R}]I$ is the intersection of all
$*$-real $*$-prime $*$-ideals of $\pcal_n$ containing
$I$ provided such a $*$-ideal exists, and
$\sqrt[\mathrm{*R}]I = \pcal_n$ otherwise.
   \end{pro}

   \begin{lem} \label{ccont}
If $\varLambda$ is a positive definite complex linear
functional on $\pcal_n$, then
$\idealfc\varLambda=\sqrt[\mathrm{*R}]
{\idealfc\varLambda}$.
   \end{lem}

   \bibliographystyle{amsalpha}
   
   \end{document}